\documentclass[12pt,twoside]{amsart}
\usepackage{amsfonts}
\usepackage{amsmath}
\usepackage{amssymb}
\usepackage{graphicx}
\usepackage{mathrsfs}
\usepackage{tikz}
\usepackage{lscape}

\newtheorem{theo}{Theorem}

\newtheorem{lem}{Lemma}

\theoremstyle{definition}

\theoremstyle{remark}
\newtheorem{rem}{\bf Remark\/}

\numberwithin{equation}{section}

\def\1{{\mathchoice {\rm 1\mskip-4mu l} {\rm 1\mskip-4mu l}{\rm 1\mskip-4.5mu l} {\rm 1\mskip-5mu l}}}
\newcommand{\ds}{\displaystyle}

\usepackage[width=16.00cm, height=21.00cm, left=2.00cm, right=2.00cm, top=4cm, bottom=2cm]{geometry}

\title{Spectral properties of the Cauchy transform on modified Bergman spaces}

\author[K. Chbichib, N. Ghiloufi and S. Snoun]{Khaled Chbichib, Noureddine Ghiloufi and Safa Snoun}
\email{c.khaled@yahoo.fr, noureddine.ghiloufi@fsg.rnu.tn, snoun.safa@fsg.rnu.tn}
\address{University of Gabes\\ Faculty of Sciences of Gabes\\  Laboratory of Mathematics and Applications (LR17ES11)\\ 6072, Gabes, Tunisia.}

\begin{document}

\begin{abstract}
    In this paper, we determine the singular values $s_n(T_{\alpha,\beta})$ and $s_n(R_{\alpha,\beta})$ of the operators $T_{\alpha,\beta}=\mathcal C\mathbb P_{\alpha,\beta}$ and $R_{\alpha,\beta}=\mathbb P_{\alpha,\beta}\mathcal C\mathbb P_{\alpha,\beta}$ where $\mathcal C$ is the integral Cauchy transform and $\mathbb P_{\alpha,\beta}$ is the orthogonal projection from $L^2(\mathbb D,\mu_{\alpha,\beta})$ onto the modified Bergman space $\mathcal A^2(\mathbb D,\mu_{\alpha,\beta})$. These singular values will be expressed in terms of some series involving hypergeometric functions. We show that in both cases the sequence $n^{\alpha+1}s_n(.)$ has a finite limit as $n\to+\infty$.
\end{abstract}
\maketitle
\section{Introduction}
Let $\mathbb D$ be the unit disc of $\mathbb C$ and let $\mu_{\alpha,\beta}$ be the probability measure on $\mathbb D$ given by $$d\mu_{\alpha,\beta}(z):=\frac{|z|^{2\beta}(1-|z|^2)^\alpha }{\mathscr B(\alpha+1,\beta+1)}dA(z)$$  where $\alpha,\beta$ are two fixed real parameters such that $\alpha,\beta>-1$. Here $\mathscr B(.,.)$ is the standard beta function and $dA(z)=\frac{1}{\pi}dxdy,\ z=x+iy$ is the normalized area measure on $\mathbb D$.\\

By $\mathcal A^2(\mathbb D,\mu_{\alpha,\beta})$ we denote the space of holomorphic functions $f$ on $\mathbb D^*:=\mathbb D\smallsetminus\{0\}$ with finite norm  $\|f\|_2=\langle f,f\rangle_{\alpha,\beta}^{\frac{1}{2}}$ where $\langle .,.\rangle_{\alpha,\beta}$ is the inner product on $L^2(\mathbb D,\mu_{\alpha,\beta})$ defined as follows:
$$\langle f,g\rangle_{\alpha,\beta}:=\int_{\mathbb D}f(z)\overline{g(z)}d\mu_{\alpha,\beta}(z),\quad \forall\; f,g\in L^2(\mathbb D,\mu_{\alpha,\beta}).$$
It is well known (see for example \cite{AG-Sn, Gh-Sn, Gh-Za}) that $\mathcal A^2(\mathbb D,\mu_{\alpha,\beta})$ is a Hilbert space and we can restrict the study to the case  $-1<\beta\leq 0$ that will be assumed throughout this paper. In this case, the reproducing kernel of $\mathcal A^2(\mathbb D,\mu_{\alpha,\beta})$ is $\mathbb K_{\alpha,\beta}(z,w)=\mathcal K_{\alpha,\beta}(z\overline{w})$ where
\begin{equation}\label{q1}
\mathcal K_{\alpha,\beta}(\xi)=\ds\ _2F_1\left(\left.
\begin{array}{cc}
1,&\alpha+\beta+2\\
& \beta+1
\end{array}\right|\xi\right).
\end{equation}
Here $\ _2F_1$ is the Gauss hypergeometric function. Indeed, many results will be expressed in terms of the hypergeometric functions:
    $$_kF_j\left(\left.
\begin{array}{c}
a_1,\dots,a_k\\
b_1,\dots,b_j
\end{array}\right|\xi\right)=\sum_{n=0}^{+\infty}\frac{(a_1)_n\dots(a_k)_n}{(b_1)_n\dots(b_j)_n}\frac{\xi^n}{n!}$$
where $(a)_n=a(a+1)\dots(a+n-1)$ is the Pochhammer symbol for every $a\in\mathbb R$. In fact  $(a)_n=\Gamma(a+n)/\Gamma(a)$  for every $a\not\in\mathbb Z^-$. (see \cite{Ma-Ob-So} for more details about hypergeometric functions).\\

It follows that the orthogonal projection $\mathbb P_{\alpha,\beta}$ from  $L^2(\mathbb D,\mu_{\alpha,\beta})$ onto $\mathcal A^2(\mathbb D,\mu_{\alpha,\beta})$ is given by the following integral transform:
$$\mathbb P_{\alpha,\beta}f(w)=\int_{\mathbb D}f(\xi)\mathcal K_{\alpha,\beta}(w\overline{\xi})d\mu_{\alpha,\beta}(\xi).$$
These modified Bergman spaces are interesting. Indeed, it is claimed in \cite{Dk-Gh-Sn} that the space $\mathcal A^2(\mathbb D,\mu_{\alpha,\beta})$ is a modification of the classical Bergman space $\mathcal A^2(\mathbb D,\mu_{\alpha})$ resulting from some physical problems. In particular, its associated Landau operator can be viewed as a magnetic Schr\"odinger operator with a \textbf{singular  potential} $\theta_{\alpha,\beta}$ given by $$\theta_{\alpha,\beta}(z)=\frac{i\alpha}{1-|z|^2}(\overline{z}dz-zd\overline{z})-i\beta\left(\frac{dz}{z}-\frac{d\overline{z}}{\overline{z}}\right).$$
We consider the two operators $T_{\alpha,\beta}=\mathcal C\mathbb  P_{\alpha,\beta}$ and $R_{\alpha,\beta}=\mathbb P_{\alpha,\beta}\mathcal C\mathbb P_{\alpha,\beta}$ where $\mathcal C$  is the Cauchy transform acting on $L^2(\mathbb D,\mu_{\alpha,\beta})$ as follows:
$$\mathcal Cf(z)=-\int_{\mathbb D}\frac{f(w)}{w-z}d\mu_{\alpha,\beta}(w).$$

The paper aims to find suitable conditions on $\alpha$ and $\beta$ for which the Cauchy transform $\mathcal C$ is a compact operator on $L^2(\mathbb D,\mu_{\alpha,\beta})$ and to find the singular values $s_n(T_{\alpha,\beta})$ and $s_n(R_{\alpha,\beta})$ of the compact operators $T_{\alpha,\beta}$ and  $R_{\alpha,\beta}$ respectively. Namely, for a compact operator $T$, the singular values $(s_n(T))_n$ is the sequence of eigenvalues of the operator $|T|=(T^*T)^{\frac{1}{2}}$.\\

Historically, the Cauchy transform $\mathcal C$ on $L^2(\Omega)$ was the aim of several papers. For a bounded domain $\Omega$ of $\mathbb C$, in 1989, Anderson and Hinkkanen proved in \cite{AH} the continuity of $\mathcal C$ and they gave an estimation of its norm for $\Omega=\mathbb D$. Indeed they showed that in this case $\|\mathcal C\|= \frac{2}{j_0}$ where $j_0$ is the smallest positive zero of the Bessel function $J_0$. One year later, this result was reproved in \cite{AKL} by Anderson, Khavinson and Lomonosov where they proved that $\mathcal C$ is compact. Between 1998 and 2000, Dostanic (see \cite{Do1, Do2}) gave some spectral properties of $\mathcal C$ on some weighted Bergman spaces. Namely, in each case, he gave the asymptotic behavior of the singular values of $\mathcal C$ and those of its product with the Bergman projection. In 2005, he studied in \cite{Do3} the norm of $\mathcal C$ on $L^2(\mathbb D,\mu)$ where $d\mu(z)=\omega(|z|)d A(z)$ for some radial weight $\omega$. Five years later, in 2010, he returned to the unweighted case on $\mathbb D$ and he gave in \cite{Do4} an exact asymptotic behavior of the singular values of $\mathcal C P_h$ where $P_h$ is the Bergman projection onto the harmonic Bergman space.\\

In our statement, we extend the Dostanic study by giving explicitly the singular values $s_n(T_{\alpha,\beta})$ and $s_n(R_{\alpha,\beta})$ of $T_{\alpha,\beta}$ and  $R_{\alpha,\beta}$ respectively and then we find their asymptotic behavior when $n$ goes to infinity.

To simplify the results, we need the function $h$ (the incomplete beta function) given by $$h(r)=\int_0^rt^\beta(1-t)^\alpha dt,\quad \forall\; r\in[0,1]$$
and the sequences $(\mathcal I_n)_{n\geq0}$ and $(\mathcal J_n)_{n\geq0}$ where
$$\mathcal I_n:=\int_0^1\frac{h(r)}{h(1)}r^{\beta+n-1}(1-r)^\alpha dr\quad\text{and}\quad \mathcal J_n:=\int_0^1\left(\frac{h(r)}{h(1)}\right)^2r^{\beta+n-1}(1-r)^\alpha dr.
$$
It is interesting to claim that if we apply Formula (18) in  \cite{CF}, we obtain
$$\mathcal I_n=\ds \frac{\mathscr B(\alpha+1,2\beta+n+1)}{(\beta+1)\mathscr B(\alpha+1,\beta+1)} \ _3F_2\left(\left.
\begin{array}{c}
-\alpha,\beta+1,2\beta+n+1\\
 \beta+2,\alpha+2\beta+n+2
\end{array}\right|1\right).
$$
While to obtain  $\mathcal J_n$, we may apply  Formula (17) in \cite{CF} but this formula contains a mistake, so to correct it we may return to Formula (16) in the same reference \cite{CF} and we continue the competition to obtain
$$\begin{array}{lcl}
\mathcal J_n&=&\ds \frac{1}{(\beta+1)\mathscr B^2(\alpha+1,\beta+1)}\sum_{k=0}^{+\infty}\frac{(-\alpha)_k}{(\beta+1+k)k!}\frac{\Gamma(3\beta+\alpha+n+k+3)}{\Gamma(3\beta+n+k+2)}\times \\
&&\ds\hfill\ _3F_2\left(\left.
\begin{array}{c}
-\alpha,\beta+1,3\beta+n+k+2\\
 \beta+2,3\beta+\alpha+n+k+3
\end{array}\right|1\right).
\end{array}$$

Now we can cite the main result of the paper:
\begin{theo}\label{th1}
For every $-\frac{1}{2}<\beta\leq 0$ and $\alpha>-1$, the two operators $T_{\alpha,\beta}$ and $R_{\alpha,\beta}$ are  well-defined and compact from $L^2(\mathbb D,\mu_{\alpha,\beta})$ into itself and their singular values are given by
$$s_n^2(T_{\alpha,\beta})=\left\{\begin{array}{lcl}
\ds\frac{\mathcal J_0}{\mathscr B(\alpha+1,\beta+1)}&if &n=0\\
\ds\frac{\Gamma(\alpha+\beta+2+n)}{\Gamma(\alpha+1)\Gamma(\beta+1+n)}\Bigl(\mathscr B(\alpha+1,\beta+n)-2\mathcal I_n+\mathcal J_n\Bigr)&if& n\geq1
\end{array}\right.
$$
and
$$s_n(R_{\alpha,\beta})= \frac{1}{\Gamma(\alpha+1)}\sqrt{\frac{\alpha+\beta+2+n}{\beta+1+n}}\frac{\Gamma(\alpha+\beta+2+n)}{\Gamma(\beta+1+n)}(\mathscr B(\alpha+1,\beta+n+1)-\mathcal I_{n+1}).$$
Moreover, we have
\begin{align*}
    s_n^2(T_{\alpha,\beta})&\underset{n\to+\infty}{\sim} \frac{1}{\mathscr B^2(\alpha+1,\beta+1)}\frac{\Gamma(3\alpha+3)}{(\alpha+1)^2\Gamma(\alpha+1)}\frac{1}{n^{2\alpha+2}}\\
    s_n(R_{\alpha,\beta})&\underset{n\to+\infty}{\sim} \frac{1}{\mathscr B(\alpha+1,\beta+1)}\frac{\Gamma(2\alpha+2)}{\Gamma(\alpha+2)}\frac{1}{n^{\alpha+1}}.
\end{align*}
\end{theo}
To prove the above theorem, we will proceed in steps. As a first step, we discuss the Cauchy transform's continuity and compactness to conclude that the operators $T_{\alpha,\beta}$ and $R_{\alpha,\beta}$ are compact. We determine in the second step the singular values $s_n(T_{\alpha,\beta})$ and $s_n(R_{\alpha,\beta})$ of $T_{\alpha,\beta}$ and $R_{\alpha,\beta}$ respectively. To achieve the proof, we find at the end of the paper the asymptotic behavior of these two singular values when $n$ goes to infinity.
\section{Compactness of $\mathcal C$}

\begin{lem}\label{l1}
    For every $-\frac{1}{2}<\beta\leq0$, the Cauchy transform $\mathcal C$ is continuous from $L^2(\mathbb D,\mu_{\alpha,\beta})$ into itself. Moreover, if $\alpha>-\frac{2\beta+1}{3\beta+2}$ then $\mathcal C$ is compact.
\end{lem}
This result is proved in \cite{AKL} for $\alpha=\beta=0$.
\begin{proof}
    If we set $\varphi(z)=\frac{1}{z}$ then it is easy to see that $\varphi\in L^q(\mathbb D,\mu_{\alpha,\beta})$ for every $q\in]0,2(\beta+1)[$ and
    $$\|\varphi\|_q=\left(\int_{\mathbb D}|\varphi(\xi)|^q d\mu_{\alpha,\beta}(\xi)\right)^{\frac{1}{q}}=\left(\frac{\mathscr B(\alpha+1,\beta+1+\frac{q}{2})}{\mathscr B(\alpha+1,\beta+1)}\right)^{\frac{1}{q}}.$$
    In particular, if $\beta>-\frac{1}{2}$ then we obtain $\varphi\in L^1(\mathbb D,\mu_{\alpha,\beta})$. It follows that for every $f\in L^2(\mathbb D,\mu_{\alpha,\beta})$,   $\mathcal Cf=f*\varphi\in L^2(\mathbb D,\mu_{\alpha,\beta}) $ as a convolution product and $\|\mathcal Cf\|_2\leq\|\varphi\|_1 \|f\|_2$. Hence $\mathcal C$ is a linear continuous operator from $L^2(\mathbb D,\mu_{\alpha,\beta})$ into itself.\\

    To discuss the compactness of  $\mathcal C$ on $L^2(\mathbb D,\mu_{\alpha,\beta})$, we will find a suitable condition on $\alpha$ for which $\mathcal C^*\mathcal C$ is a Hilbert-Schmidt operator. Using Fubini theorem, for every $f,g\in L^2(\mathbb D,\mu_{\alpha,\beta})$, we have
    $$\begin{array}{lcl}
    \langle\mathcal C f,g\rangle_{\alpha,\beta}&=&\ds\int_{\mathbb D}\mathcal Cf(\xi)\overline{g(\xi)}d\mu_{\alpha,\beta}(\xi)\\
    &=&-\ds\int_{\mathbb D}\left(\int_{\mathbb D} \frac{f(w)}{w-\xi}d\mu_{\alpha,\beta}(w)\right) \overline{g(\xi)}d\mu_{\alpha,\beta}(\xi)\\
    &=&\ds -\int_{\mathbb D}f(w)\left(\int_{\mathbb D} \frac{\overline{g(\xi)}}{w-\xi}d\mu_{\alpha,\beta}(\xi)\right) d\mu_{\alpha,\beta}(w).
    \end{array}
    $$
    That gives
    $$\mathcal C^*g(w)=-\int_{\mathbb D} \frac{g(\xi)}{\overline{w}-\overline{\xi}}d\mu_{\alpha,\beta}(\xi).$$
    It follows that
    $$\begin{array}{lcl}
    \ds\mathcal C^*\mathcal Cf(z)&=&\ds-\int_{\mathbb D} \frac{1}{\overline{z}-\overline{\xi}}\mathcal Cf(\xi)d\mu_{\alpha,\beta}(\xi)\\
    &=&\ds\int_{\mathbb D}\frac{1}{\overline{z}-\overline{\xi}}\left(\int_{\mathbb D} \frac{f(w)}{w-\xi}d\mu_{\alpha,\beta}(w)\right)d\mu_{\alpha,\beta}(\xi)\\
    &=&\ds\int_{\mathbb D}f(w)\left(\int_{\mathbb D}    \frac{1}{(\overline{z}-\overline{\xi})(w-\xi)} d\mu_{\alpha,\beta}(\xi)\right)d\mu_{\alpha,\beta}(w)\\
    &=&\ds \int_{\mathbb D}f(w)\kappa(w,z)d\mu_{\alpha,\beta}(w).
    \end{array}$$
    To prove that $\mathcal C^*\mathcal C$ is a Hilbert-Schmidt, it suffices to show that $\kappa\in L^2(\mathbb D\times \mathbb D,\mu_{\alpha,\beta}\otimes \mu_{\alpha,\beta})$. To reach this aim, we take $z\in\mathbb D^*$ and we consider the function $\psi_z(\xi)=\frac{1}{\overline{z}-\overline{\xi}}$. Then one has $\kappa(w,z)=\psi_z*\varphi(w)$. Now, we claim that $\psi_z\in L^p(\mathbb D,\mu_{\alpha,\beta})$ for every $p\in]0,2[$ and $\varphi\in L^q(\mathbb D,\mu_{\alpha,\beta})$ for every $q\in]0,2(\beta+1)[$. We will show that we can choose $p$ and  $q$ such that $\frac{1}{p}+\frac{1}{q}=\frac{3}{2}$. This yields to $p\in]\frac{2(\beta+1)}{3\beta+2},2[$ so that the corresponding $q$ is in $]1,2(\beta+1)[$. Thanks to Young theorem, one can conclude that with this choice $\kappa(.,z)=\psi_z*\varphi\in L^2(\mathbb D,\mu_{\alpha,\beta})$. Thus using the Fubini-Tonelli theorem,
    $$\|\kappa\|^2_{L^2(\mathbb D\times \mathbb D,\mu_{\alpha,\beta}\otimes \mu_{\alpha,\beta})} =\int_{\mathbb D}\left(\int_{\mathbb D}|\kappa(\xi,z)|^2d\mu_{\alpha,\beta}(\xi)\right) d\mu_{\alpha,\beta}(z)\leq \|\varphi\|_q^2\int_{\mathbb D} \|\psi_z\|_p^2d\mu_{\alpha,\beta}(z).
    $$
    To compute the last integral, we distinguish two cases:
    \begin{enumerate}
        \item First case: $-1<\alpha<0$.  Let $k,s>1$ be two conjugate real numbers ($\frac{1}{k}+\frac{1}{s}=1$) such that $\beta s>-1,\alpha s>-1 $. This gives
        $$1<s<\min\left(-\frac{1}{\beta},-\frac{1}{\alpha}\right),\quad k>\max\left(\frac{1}{\beta+1},\frac{1}{\alpha+1}\right).$$
        The two exponents $k$ and $s$ are chosen to apply the H\"older inequality as follows
        \begin{align*}
            \|\psi_z\|_p^p&=\frac{1}{h(1)}\int_{\mathbb D}\frac{1}{|z-\xi|^p}|\xi|^{2\beta}(1-|\xi|^2)^\alpha dA(\xi)\\
            &\leq \frac{1}{h(1)}\left(\int_{\mathbb D}\frac{1}{|z-\xi|^{pk}} dA(\xi)\right)^{\frac{1}{k}}\left(\int_{\mathbb D}|\xi|^{2\beta s}(1-|\xi|^2)^{\alpha s} dA(\xi)\right)^{\frac{1}{s}}\\
            &\leq \frac{\mathscr B(\alpha s+1,\beta s+1)^{\frac{1}{s}}}{\mathscr B(\alpha+1,\beta+1)} \left(\int_{\mathbb D(z,1+|z|)}\frac{1}{|z-\xi|^{pk}} dA(\xi)\right)^{\frac{1}{k}}\\
            &\leq \frac{\mathscr B(\alpha s+1,\beta s+1)^{\frac{1}{s}}}{\mathscr B(\alpha+1,\beta+1)} \left(\frac{(1+|z|)^{2-pk}}{2-pk}\right)^{\frac{1}{k}}
        \end{align*}
        whenever $pk<2$. Hence we obtain $\|\kappa\|^2_{L^2(\mathbb D\times \mathbb D,\mu_{\alpha,\beta}\otimes \mu_{\alpha,\beta})}$ is finite.  It follows that $\mathcal C^*\mathcal C$ is a Hilbert-Schmidt operator from $L^2(\mathbb D,\mu_{\alpha,\beta})$ into itself and so $\mathcal C$ is a compact operator from $L^2(\mathbb D,\mu_{\alpha,\beta})$ into itself if all assumptions on $p$ are satisfied. To conclude the result, it suffices to show that $]\frac{2(\beta+1)}{3\beta+2},2[\cap ]0,2\min(\beta+1,\alpha+1)[\neq \emptyset.$ This is fulfilled if $\alpha>-\frac{2\beta+1}{3\beta+2}$.
        \item Second case: $\alpha\geq0$. Again one can choose two conjugate real numbers $k,s$ such that $\beta s>-1$ to obtain
        \begin{align*}
            \|\psi_z\|_p^p&\leq\frac{1}{h(1)}\int_{\mathbb D}\frac{1}{|z-\xi|^p}|\xi|^{2\beta} dA(\xi)\\
            &\leq \frac{1}{h(1)}\left(\int_{\mathbb D}\frac{1}{|z-\xi|^{pk}} dA(\xi)\right)^{\frac{1}{k}}\left(\int_{\mathbb D}|\xi|^{2\beta s} dA(\xi)\right)^{\frac{1}{s}}\\
            &\leq \frac{1}{(\beta s+1)^{\frac{1}{s}}h(1)} \left(\int_{\mathbb D(z,1+|z|)}\frac{1}{|z-\xi|^{pk}} dA(\xi)\right)^{\frac{1}{k}}\\
            &\leq \frac{1}{(\beta s+1)^{\frac{1}{s}}h(1)} \left(\frac{(1+|z|)^{2-pk}}{2-pk}\right)^{\frac{1}{k}}.
        \end{align*}
        This is possible since $-\frac{1}{2}<\beta\leq0$ and we can proceed as in the previous case to conclude that $\mathcal C$ is compact.
    \end{enumerate}
\end{proof}
\section{Singular values of operators}
\subsection{Singular values of $T_{\alpha,\beta}$}
For instance, if we take $\alpha>-\frac{2\beta+1}{3\beta+2}$ then as an immediate consequence of Lemma \ref{l1} and since $\mathbb P_{\alpha,\beta}$ is continuous and $\mathcal C$ is compact on $L^2(\mathbb D,\mu_{\alpha,\beta})$ then $T_{\alpha,\beta}$ is also compact. The condition $\alpha>-\frac{2\beta+1}{3\beta+2}$ will be omitted and replaced by $\alpha>-1$ after the determination of the sequence of singular values of $T_{\alpha,\beta}$. To reach this aim,  we start by finding the reproducing kernel of $T_{\alpha,\beta}$ and then we evaluate this kernel to decompose $T_{\alpha,\beta}$ using two orthonormal sequences in  $L^2(\mathbb D,\mu_{\alpha,\beta})$.
\begin{lem}\label{l2}
    For every $\alpha>-1$ and $f\in L^2(\mathbb D,\mu_{\alpha,\beta})$,
    $$T_{\alpha,\beta}f(z)=\int_{\mathbb D}f(w)\tau_{\alpha,\beta}(z,\xi)d\mu_{\alpha,\beta}(\xi)$$
    where $$\tau_{\alpha,\beta}(z,\xi)=\frac{1}{z}+\left(\frac{h(|z|^2)}{h(1)}-1\right)\frac{\mathcal K_{\alpha,\beta}(z\overline{\xi})}{z}.$$
\end{lem}
\begin{proof}
  Let $\alpha>-1$ and $f\in L^2(\mathbb D,\mu_{\alpha,\beta})$. Then as mentioned before,
    $$\mathbb P_{\alpha,\beta}f(w)=\int_{\mathbb D}f(\xi)\mathcal K_{\alpha,\beta}(w\overline{\xi})d\mu_{\alpha,\beta}(\xi).$$
Hence,
$$\begin{array}{lcl}
\ds T_{\alpha,\beta}f(z)&=&\ds-\int_{\mathbb D}\frac{\mathbb P_{\alpha,\beta}f(w)}{w-z}d\mu_{\alpha,\beta}(w) =-\int_{\mathbb D}\frac{1}{w-z}\left(\int_{\mathbb D}f(\xi)\mathcal K_{\alpha,\beta}(w\overline{\xi})d\mu_{\alpha,\beta}(\xi)\right)d\mu_{\alpha,\beta}(w)\\
&=&\ds -\int_{\mathbb D} f(\xi)\left(\int_{\mathbb D}  \frac{\mathcal K_{\alpha,\beta}(w\overline{\xi})}{w-z}d\mu_{\alpha,\beta}(w)\right)d\mu_{\alpha,\beta}(\xi) =\int_{\mathbb D} f(\xi)\tau_{\alpha,\beta}(z,\xi)d\mu_{\alpha,\beta}(\xi).
\end{array}$$
To conclude the result, it suffices to determine the kernel
$$\tau_{\alpha,\beta}(z,\xi)=-\int_{\mathbb D}  \frac{\mathcal K_{\alpha,\beta}(w\overline{\xi})}{w-z}d\mu_{\alpha,\beta}(w).$$
The idea is to use Stokes' formula. To this aim, as mentioned above,  we consider the function $h$ defined on $[0,1]$ by:
$$h(r)=\int_0^rt^\beta(1-t)^\alpha dt.$$
Then $h$ is continuous on $[0,1]$ and $\mathcal C^1$ on $]0,1[$ with derivative $h'(r)=r^\beta(1-r)^\alpha$.\\
Let $\xi\in\mathbb D,\ z\in\mathbb D^*$ be some fixed points and $g_{z,\xi}$ be the function defined on $\mathbb D\smallsetminus \{0,z\}$ by $$g_{z,\xi}(w)=-\frac{\mathcal K_{\alpha,\beta}(w\overline{\xi})}{w(w-z)}\frac{h(|w|^2)}{h(1)}.$$ Then $g_{z,\xi}$ is $\mathcal C^1$ on $\mathbb D\smallsetminus \{0,z\}$. For every $\varepsilon,\ \tau$ and $\eta$ such that $0<\varepsilon<|z|-\tau<|z|+\tau<\eta<1$, we set $\Omega_{\varepsilon,\tau,\eta}:=\mathbb D(0,\eta)\smallsetminus\Bigl(\mathbb D(0,\varepsilon)\cup\mathbb D(z,\tau)\Bigr)$. Thanks to Stokes' formula,
\begin{equation}\label{2.1}
  \begin{array}{l}
   -\ds\int_{\Omega_{\varepsilon,\tau,\eta}}  \frac{\mathcal K_{\alpha,\beta}(w\overline{\xi})}{w-z}d\mu_{\alpha,\beta}(w)=\ds -\frac{1}{2i\pi} \int_{\Omega_{\varepsilon,\tau,\eta}}\frac{\partial g_{z,\xi}}{\partial \overline{w}}(w)dw\wedge d\overline{w}\\
   = \ds\frac{1}{2i\pi} \int_{\partial \mathbb D(0,\eta)}g_{z,\xi}(w)dw-\frac{1}{2i\pi} \int_{\partial \mathbb D(0,\varepsilon)}g_{z,\xi}(w)dw-\frac{1}{2i\pi} \int_{\partial \mathbb D(z,\tau)}g_{z,\xi}(w)dw\\
   =\mathcal N_1(\eta)+\mathcal N_2(\varepsilon)+\mathcal N_3(\tau).
\end{array}
\end{equation}
Now, thanks to Residues theorem,
$$
\begin{array}{lcl}
    \ds\mathcal N_1(\eta)&=&\ds \frac{1}{2i\pi} \int_{\partial \mathbb D(0,\eta)}g_{z,\xi}(w)dw =-\frac{h(\eta^2)}{h(1)}\frac{1}{2i\pi} \int_{\partial \mathbb D(0,\eta)} \frac{\mathcal K_{\alpha,\beta}(w\overline{\xi})}{w(w-z)}dw\\
    &=&\ds \frac{h(\eta^2)}{h(1)}\frac{1-\mathcal K_{\alpha,\beta}(z\overline{\xi})}{z}.
\end{array}$$
Thus we find
\begin{equation}\label{2.2}
    \lim_{\eta\to1^-}\mathcal N_1(\eta)=\frac{1-\mathcal K_{\alpha,\beta}(z\overline{\xi})}{z}.
\end{equation}
Again for $\mathcal N_2(\varepsilon)$, we obtain
$$
\begin{array}{lcl}
    \ds\mathcal N_2(\varepsilon)&=&\ds -\frac{1}{2i\pi} \int_{\partial \mathbb D(0,\varepsilon)}g_{z,\xi}(w)dw =\frac{h(\varepsilon^2)}{h(1)}\frac{1}{2i\pi} \int_{\partial \mathbb D(0,\varepsilon)} \frac{\mathcal K_{\alpha,\beta}(w\overline{\xi})}{w(w-z)}dw\\
    &=&\ds -\frac{h(\varepsilon^2)}{h(1)}\frac{1}{z}.
\end{array}$$
Using the fact that $h(0)=0$, we obtain
\begin{equation}\label{2.3}
    \lim_{\varepsilon\to0^+}\mathcal N_2(\varepsilon)=0.
\end{equation}
However, for  $\mathcal N_3(\tau)$,
$$
\begin{array}{lcl}
    \ds\mathcal N_3(\tau)&=&\ds -\frac{1}{2i\pi} \int_{\partial \mathbb D(z,\tau)}g_{z,\xi}(w)dw =\frac{1}{h(1)}\frac{1}{2i\pi} \int_{\partial \mathbb D(z,\tau)} \frac{\mathcal K_{\alpha,\beta}(w\overline{\xi})}{w(w-z)}h(|w|^2)dw\\
    &=&\ds \frac{1}{h(1)}\frac{1}{2i\pi} \int_{\partial \mathbb D(z,\tau)} \frac{\mathcal K_{\alpha,\beta}(w\overline{\xi})}{w(w-z)}(h(|w|^2)-h(|z|^2)dw +\frac{h(|z|^2)}{h(1)}\frac{\mathcal K_{\alpha,\beta}(z\overline{\xi})}{z}.
\end{array}$$
Using the continuity of $w\longmapsto h(|w|^2)$ at $z$, one can conclude that
\begin{equation}\label{2.4}
    \lim_{\tau\to0^+}\mathcal N_3(\tau)=\frac{h(|z|^2)}{h(1)}\frac{\mathcal K_{\alpha,\beta}(z\overline{\xi})}{z}.
\end{equation}
Combining \eqref{2.2}, \eqref{2.3} and \eqref{2.4} with \eqref{2.1}, we obtain
$$\tau_{\alpha,\beta}(z,\xi)=-\int_{\mathbb D}  \frac{\mathcal K_{\alpha,\beta}(w\overline{\xi})}{w-z}d\mu_{\alpha,\beta}(w)=\frac{1}{z}+\left(\frac{h(|z|^2)}{h(1)}-1\right)\frac{\mathcal K_{\alpha,\beta}(z\overline{\xi})}{z}
$$
and the desired result is proved.
\end{proof}
Now we can decompose $T_{\alpha,\beta}$ as a sum using two orthonormal sequences in $L^2(\mathbb D,\mu_{\alpha,\beta})$. Indeed, if we set $$e_n(w)=\sqrt{\frac{(\alpha+\beta+2)_n}{(\beta+1)_n}}w^n,$$
then the sequence $(e_n)_{n\geq 0}$ is a Hilbert basis of $\mathcal A^2(\mathbb D,\mu_{\alpha,\beta})$. Hence, using Lemma \ref{l2}  we find that the kernel of the operator  $T_{\alpha,\beta}$ is
$$\begin{array}{lcl}
\tau_{\alpha,\beta}(z,\xi)&=&\ds \frac{1}{z}+\left(\frac{h(|z|^2)}{h(1)}-1\right) \frac{\mathcal K_{\alpha,\beta}(z\overline{\xi})}{z}\\
&=&\ds \frac{h(|z|^2)}{zh(1)}+\left(\frac{h(|z|^2)}{h(1)}-1\right)\sum_{n=1}^{+\infty}\frac{(\alpha+\beta+2)_n}{(\beta+1)_n}z^{n-1}\overline{\xi}^n\\
&=&\ds \frac{h(|z|^2)}{zh(1)}+\left(\frac{h(|z|^2)}{h(1)}-1\right)\sum_{n=1}^{+\infty}\sqrt{\frac{(\alpha+\beta+2)_n}{(\beta+1)_n}}z^{n-1}e_n(\overline{\xi}).
\end{array}
$$

If we set $d_n(\alpha,\beta)$ such that

$$d_n^2(\alpha,\beta):=
\left\{\begin{array}{lcl}
\ds\frac{1}{h^3(1)}\int_0^1(h(r))^2r^{\beta-1}(1-r)^\alpha dr&if& n=0\\
\ds\frac{1}{h^3(1)}\int_0^1(h(1)-h(r))^2r^{\beta+n-1}(1-r)^\alpha dr&if &n\geq 1.
\end{array}\right.$$
and

$$\psi_n(z)=
\left\{\begin{array}{lcl}
\ds \frac{1}{d_0(\alpha,\beta)}\frac{h(|z|^2)}{zh(1)}&if &n=0\\
\ds \frac{1}{d_n(\alpha,\beta)}\left(\frac{h(|z|^2)}{h(1)}-1\right)z^{n-1}&if & n\geq 1
\end{array}\right.$$
then it is easy to see that $(\psi_n)_{n\geq 0}$ is an orthonormal sequence in $L^2(\mathbb D,\mu_{\alpha,\beta})$.\\
With these notations, we obtain
$$\tau_{\alpha,\beta}(z,\xi)=\ds \sum_{n=0}^{+\infty}d_n(\alpha,\beta)\sqrt{\frac{(\alpha+\beta+2)_n}{(\beta+1)_n}}\psi_n(z)e_n(\overline{\xi}).
$$
It follows that for every $f\in L^2(\mathbb D,\mu_{\alpha,\beta})$,
$$\begin{array}{lcl}
T_{\alpha,\beta}f(z)&=&\ds \sum_{n=0}^{+\infty}d_n(\alpha,\beta)\sqrt{\frac{(\alpha+\beta+2)_n}{(\beta+1)_n}}\psi_n(z)\int_{\mathbb D}f(\xi)e_n(\overline{\xi})d\mu_{\alpha,\beta}(\xi)\\
&=&\ds \sum_{n=0}^{+\infty}d_n(\alpha,\beta) \sqrt{\frac{(\alpha+\beta+2)_n}{(\beta+1)_n}}\psi_n(z)\langle f,e_n\rangle_{\alpha,\beta}.
\end{array}
$$
We deduce so that the singular values $s_n(T_{\alpha,\beta})$ of $T_{\alpha,\beta}$ are given by $$s_n(T_{\alpha,\beta})=d_n(\alpha,\beta)\sqrt{\frac{(\alpha+\beta+2)_n}{(\beta+1)_n}}.$$
\subsection{Singular values of $R_{\alpha,\beta}$}
With the same techniques used for $T_{\alpha,\beta}$, it is easy to see that the kernel of $R_{\alpha,\beta}$ is given by
$$\varrho_{\alpha,\beta}(z,\xi)=\int_{\mathbb D}\tau_{\alpha,\beta}(w,\xi)\mathcal K_{\alpha,\beta}(z\overline{w})d\mu_{\alpha,\beta}(w).
$$
Using Lemma \ref{l2}, we deduce that
$$\begin{array}{lcl}
    \varrho_{\alpha,\beta}(z,\xi)&=&\ds\int_{\mathbb D}\left(\frac{1}{w}+\left(\frac{h(|w|^2)}{h(1)}-1\right)\frac{\mathcal K_{\alpha,\beta}(w\overline{\xi})}{w}\right)\mathcal K_{\alpha,\beta}(z\overline{w})d\mu_{\alpha,\beta}(w).\\
    &=&\ds\frac{1}{h^2(1)}\sum_{n,m=0}^{+\infty}\frac{(\alpha+\beta+2)_n(\alpha+\beta+2)_m}{(\beta+1)_n(\beta+1)_m}z^n\overline{\xi}^m\int_{\mathbb D}h(|w|^2)\overline{w}^nw^{m-1}|w|^{2\beta}(1-|w|^2)^\alpha dA(w)\\
    &&\ds-\frac{1}{h(1)}\sum_{n=0,m=1}^{+\infty}\frac{(\alpha+\beta+2)_n(\alpha+\beta+2)_m}{(\beta+1)_n(\beta+1)_m}z^n\overline{\xi}^m\int_{\mathbb D}\overline{w}^nw^{m-1}|w|^{2\beta}(1-|w|^2)^\alpha dA(w)\\
    &=&\ds\frac{1}{h^2(1)}\sum_{n=0}^{+\infty}\frac{(\alpha+\beta+2)_n(\alpha+\beta+2)_{n+1}}{(\beta+1)_n(\beta+1)_{n+1}}z^n\overline{\xi}^{n+1} \int_{\mathbb D}h(|w|^2)|w|^{2\beta+2n}(1-|w|^2)^\alpha dA(w)\\
    &&\ds-\frac{1}{h(1)}\sum_{n=0}^{+\infty}\frac{(\alpha+\beta+2)_n(\alpha+\beta+2)_{n+1}}{(\beta+1)_n(\beta+1)_{n+1}}z^n\overline{\xi}^{n+1} \int_{\mathbb D}|w|^{2\beta+2n}(1-|w|^2)^\alpha dA(w)\\
    &=&\ds \frac{1}{h^2(1)}\sum_{n=0}^{+\infty}\frac{(\alpha+\beta+2)_n(\alpha+\beta+2)_{n+1}}{(\beta+1)_n(\beta+1)_{n+1}}z^n\overline{\xi}^{n+1} \int_0^1(h(r)-h(1))r^{\beta+n}(1-r)^\alpha dr\\
    &=&\ds \frac{1}{h^2(1)}\sum_{n=0}^{+\infty}\sqrt{\frac{(\alpha+\beta+2)_n(\alpha+\beta+2)_{n+1}}{(\beta+1)_n(\beta+1)_{n+1}}} e_n(z)e_{n+1}(\overline{\xi})\int_0^1(h(r)-h(1))r^{\beta+n}(1-r)^\alpha dr.
    \end{array}
    $$
    It follows that the singular values of $R_{\alpha,\beta}$ are
    \begin{align*}
        s_n(R_{\alpha,\beta})&=\frac{1}{h^2(1)}\sqrt{\frac{(\alpha+\beta+2)_n(\alpha+\beta+2)_{n+1}}{(\beta+1)_n(\beta+1)_{n+1}}} \int_0^1(h(r)-h(1))r^{\beta+n}(1-r)^\alpha dr\\
        &=\frac{1}{h(1)}\sqrt{\frac{\alpha+\beta+2+n}{\beta+1+n}} \frac{(\alpha+\beta+2)_n}{(\beta+1)_n}(\mathscr B(\alpha+1,\beta+n+1)-\mathcal I_{n+1})\\
        &= \frac{1}{\Gamma(\alpha+1)}\sqrt{\frac{\alpha+\beta+2+n}{\beta+1+n}}\frac{\Gamma(\alpha+\beta+2+n)}{\Gamma(\beta+1+n)}(\mathscr B(\alpha+1,\beta+n+1)-\mathcal I_{n+1}).
    \end{align*}
    This is due to the equality     $$R_{\alpha,\beta}f(z)=-\sum_{n=0}^{+\infty}s_n(R_{\alpha,\beta})\langle f,e_{n+1}\rangle_{\alpha,\beta}\ e_n(z).$$
\section{Asymptotic results}
To finish the proof of the main result we shall find the asymptotic behaviors of both $s_n(T_{\alpha,\beta})$ and $s_n(R_{\alpha,\beta})$  when $n\longrightarrow+\infty$. We will see that both of them use the Mellin transform of a power of the function $h(1)-h(r)$. For this reason, we  consider the function $u_p$ defined on $[0,1]$ by $u_p(r)=(h(1)-h(r))^p$ with $p>0$ and its Mellin transform $\mathcal M_{u_p}$ defined on $]0,+\infty[$ by
$$\mathcal M_{u_p}(\eta)=\int_0^1u_p(r)r^{\eta-1}dr.$$
Indeed, using an integration by parts, we obtain
\begin{align*}
    s_{n+1}^2(T_{\alpha,\beta})&=\frac{(\alpha+\beta+2)_{n+1}}{h^3(1)(\beta+1)_{n+1}}\int_0^1(h(1)-h(r))^2r^{\beta+n}(1-r)^\alpha dr\\
         &=\frac{n(\alpha+\beta+2)_{n+1}}{3h^3(1)(\beta+1)_{n+1}}\int_0^1(h(1)-h(r))^3r^{n-1} dr\\
         &=\frac{n(\alpha+\beta+2)_{n+1}}{3h^3(1)(\beta+1)_{n+1}}\mathcal M_{u_3}(n)
\end{align*}
and
\begin{align*}
    s_n(R_{\alpha,\beta})&=\frac{1}{h^2(1)}\sqrt{\frac{(\alpha+\beta+2)_n(\alpha+\beta+2)_{n+1}}{(\beta+1)_n(\beta+1)_{n+1}}} \int_0^1(h(r)-h(1))r^{\beta+n}(1-r)^\alpha dr\\
        &=\frac{n}{2h^2(1)}\sqrt{\frac{(\alpha+\beta+2)_n(\alpha+\beta+2)_{n+1}}{(\beta+1)_n(\beta+1)_{n+1}}} \int_0^1(h(r)-h(1))^2r^{n-1}dr\\
        &=\frac{n}{2h^2(1)}\sqrt{\frac{(\alpha+\beta+2)_n(\alpha+\beta+2)_{n+1}}{(\beta+1)_n(\beta+1)_{n+1}}}\mathcal M_{u_2}(n).
\end{align*}
Thus, it suffices to find the asymptotic behavior of $\mathcal M_{u_p}(\eta)$ for every $p>0$.\\
Since $$\frac{h(1)-h(r)}{(1-r)^{\alpha+1}}\underset{r\to1^-}{\sim}\frac{r^\beta}{\alpha+1}$$
then
$$u_p(r)\underset{r\to1^-}{\sim}\left(\frac{r^\beta}{\alpha+1}(1-r)^{\alpha+1}\right)^p.$$
That is for $\varepsilon>0$, there exists $r_0\in]0,1[$ such that for every $r\in[r_0,1]$,
\begin{equation}\label{q4.1}
    (1-\varepsilon)\left(\frac{r^\beta}{\alpha+1}(1-r)^{\alpha+1}\right)^p\leq u_p(r)\leq (1+\varepsilon)\left(\frac{r^\beta}{\alpha+1}(1-r)^{\alpha+1}\right)^p.
\end{equation}

Now, for every $\eta>0$ we set
$$x_p(\eta)=\int_{r_0}^1u_p(r)r^{\eta-1}dr,\quad y_p(\eta)=\int_0^{r_0}u_p(r)r^{\eta-1}dr $$
so that $\mathcal M_{u_p}(\eta)=x_p(\eta)+y_p(\eta)$.
Using Equation \eqref{q4.1}, we obtain
\begin{equation}\label{q4.2}
    \frac{1-\varepsilon}{(\alpha+1)^p}\int_{r_0}^1r^{p\beta+\eta-1}(1-r)^{p(\alpha+1)}dr\leq x_p(\eta)\leq \frac{1+\varepsilon}{(\alpha+1)^p}\int_{r_0}^1r^{p\beta+\eta-1}(1-r)^{p(\alpha+1)}dr.
\end{equation}
Since
$$b_p(\eta):=\int_0^{r_0}r^{p\beta+\eta-1}(1-r)^{p(\alpha+1)}dr\leq \frac{r_0^{p\beta+\eta-1}}{p(\alpha+1)+1}$$
then $b_p(\eta)=O(r_0^\eta)$ for $\eta$ large. While by Stirling formula,
$$\int_0^1r^{p\beta+\eta-1}(1-r)^{p(\alpha+1)}dr=\mathscr B(p(\alpha+1)+1,p\beta+\eta)\underset{\eta\to+\infty}{\sim} \frac{\Gamma(p(\alpha+1)+1)}{\eta^{p(\alpha+1)+1}}.$$
Thus we deduce that
$$\int_{r_0}^1r^{p\beta+\eta-1}(1-r)^{p(\alpha+1)}dr\underset{\eta\to+\infty}{\sim}\frac{\Gamma(p(\alpha+1)+1)}{\eta^{p(\alpha+1)+1}}
$$
and Equation \eqref{q4.2} gives
$$x_p(\eta)\underset{\eta\to+\infty}{\sim} \frac{1}{(\alpha+1)^p}\frac{\Gamma(p(\alpha+1)+1)}{\eta^{p(\alpha+1)+1}}.
$$
Now for $y_p(\eta)$, we have
$$y_p(\eta)=\int_0^{r_0}u_p(r)r^{\eta-1}dr \leq h^p(1)\frac{r_0^\eta}{\eta}=o(r_0^\eta)=o\left(\frac{1}{\eta^{p(\alpha+1)+1}}\right).$$
We conclude that
$$\mathcal M_{u_p}(\eta)=x_p(\eta)+y_p(\eta)\underset{\eta\to+\infty}{\sim} \frac{1}{(\alpha+1)^p}\frac{\Gamma(p(\alpha+1)+1)}{\eta^{p(\alpha+1)+1}}.$$
As an immediate application, we find that
$$s_{n+1}^2(T_{\alpha,\beta})=\frac{n(\alpha+\beta+2)_{n+1}}{3h^3(1)(\beta+1)_{n+1}}\mathcal M_{u_3}(n)\underset{n\to+\infty}{\sim} \frac{1}{\mathscr B^2(\alpha+1,\beta+1)}\frac{\Gamma(3\alpha+3)}{(\alpha+1)^2\Gamma(\alpha+1)}\frac{1}{n^{2\alpha+2}}$$
and
$$ s_n(R_{\alpha,\beta})=\frac{n}{2h^2(1)}\sqrt{\frac{(\alpha+\beta+2)_n(\alpha+\beta+2)_{n+1}}{(\beta+1)_n(\beta+1)_{n+1}}}\mathcal M_{u_2}(n)\underset{n\to+\infty}{\sim} \frac{1}{\mathscr B(\alpha+1,\beta+1)}\frac{\Gamma(2\alpha+2)}{\Gamma(\alpha+2)}\frac{1}{n^{\alpha+1}}.
$$
This finishes the proof of the main result.\\
To illustrate the previous estimates, we give here some numerical results. We start by collecting some numerical values of $s_n(T_{\alpha,\beta})$ and $s_n(R_{\alpha,\beta})$ with their approximate values $$\tilde{s}_n(T_{\alpha,\beta})=\frac{1}{(\alpha+1)\mathscr B(\alpha+1,\beta+1)}\sqrt{\frac{\Gamma(3\alpha+3)}{\Gamma(\alpha+1)}}\frac{1}{n^{\alpha+1}}$$ and $$\tilde{s}_n(R_{\alpha,\beta})=\frac{1}{\mathscr B(\alpha+1,\beta+1)}\frac{\Gamma(2\alpha+2)}{\Gamma(\alpha+2)}\frac{1}{n^{\alpha+1}}.$$ Then we point out these values to obtain the two figures \ref{fig1} and \ref{fig2}. 

\begin{tabular}{|c|c|c|}
    \hline
   $n$  & $s_n(T_{0.5,-0.5})$& $\tilde{s}_n(T_{0.5,-0.5})$\\
   \hline
    2 &0.4206514145 & 0.5436176218\\
    3&0.2413494585& 0.2959079530\\
    4&0.1625945876& 0.1921978534\\
    5&0.1193856181& 0.1375255889\\
    6&0.09256482959& 0.1046192601\\
    7&0.07453354242& 0.08302166898\\
    8&0.06170976400& 0.06795220272\\
    9&0.05219862414& 0.05694751211\\
    10&0.04491072702& 0.04862263825\\
    11&0.03917885115& 0.04214533324\\
    12&0.03457344319& 0.03698849413\\
    13&0.03080655199& 0.03280374984\\
    14&0.02767857985& 0.02935259257\\
    15&0.02504721602& 0.02646681192\\
    16&0.02280851986& 0.02402473167\\
    17&0.02088497665& 0.02193638894\\
    18&0.01921771639& 0.02013398599\\
    19&0.01776130543& 0.01856556277\\
    20&0.01648017596& 0.01719069861\\
    21&0.01534612597& 0.01597752765\\
    \hline
\end{tabular}
\hfill
\begin{tabular}{|c|c|c|}
   \hline
   $n$  & $s_n(R_{0.5,-0.5}$& $\tilde{s}_n(R_{0.5,-0.5})$\\
   \hline
    1 &0.3250006690 &0.9577979850 \\
    2&0.1767674033& 0.3386327251\\
    3&0.1153593192&0.1843283081 \\
    4&0.08292304252&0.1197247481 \\
    5&0.06332202138&0.08566805611 \\
    6&0.05040561060&0.06516989832 \\
    7&0.04136102976&0.05171623009 \\
    8&0.03473553702&0.04232909063 \\
    9&0.02970998892&0.03547399944 \\
    10&0.02579068699&0.03028823171 \\
    11&0.02266406397&0.02625335985 \\
    12&0.02012233453&0.02304103852 \\
    13&0.01802289947&0.02043425886 \\
    14&0.01626496862&0.01828444850 \\
    15&0.01477548936&0.01648682509 \\
    \hline
\end{tabular}
\vskip1cm
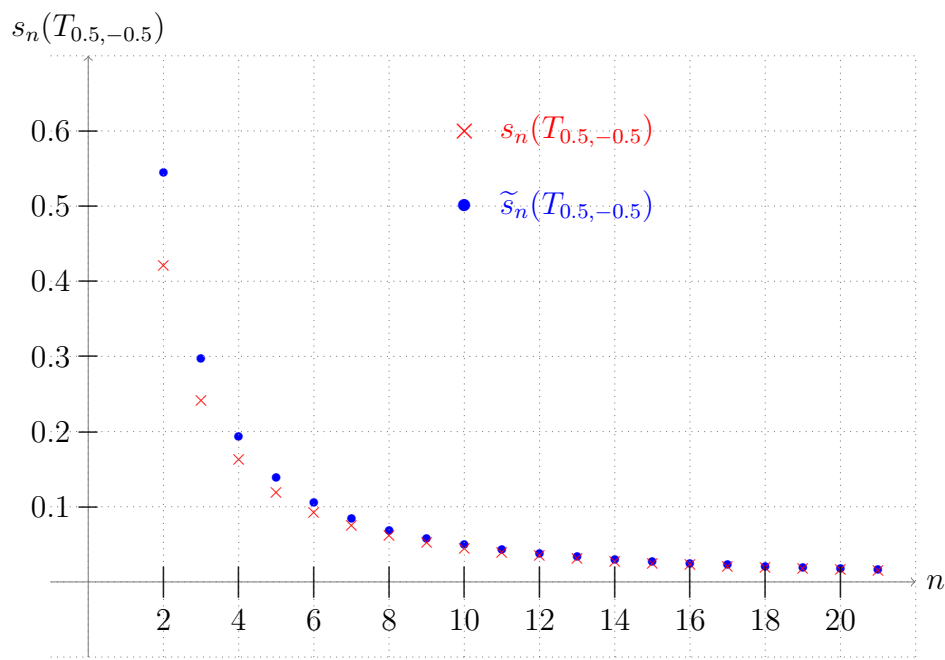
\begin{figure}[h!]
\begin{tikzpicture}[x=0.5cm,y=10cm]
 \draw[->, thin, draw=gray] (-1,0)--(22,0) node [right] {$n$}; 
 \draw[->, thin, draw=gray] (0,-0.1)--(0,0.7) node [above] {$s_n(T_{0.5,-0.5})$}; 
\foreach \Point in {(2,0.5436176218),(3, 0.2959079530),(4,0.1921978534),(5,0.1375255889),(6,0.1046192601),(7,0.08302166898),(8,0.06795220272),(9,0.05694751211),(10,0.04862263825),(11,0.04214533324),(12,0.03698849413),(13,0.03280374984),(14,0.02935259257),(15,0.02646681192),(16,0.02402473167),(17,0.02193638894),(18,0.02013398599),(19,0.01856556277),(20,0.01719069861),(21,0.01597752765)}{\node[blue, scale=0.7]  at \Point {\textbullet};
}
\foreach \Point in {(2,0.4206514145),(3,0.2413494585),(4,0.1625945876),(5,0.1193856181),(6,0.09256482959), (7,0.07453354242),(8,0.06170976400),(9,0.05219862414),(10,0.04491072702),(11,0.03917885115),(12,0.03457344319),(13,0.03080655199),(14,0.02767857985),(15,0.02504721602),(16,0.02280851986),(17,0.02088497665),(18,0.01921771639),(19,0.01776130543),(20,0.01648017596),(21,0.01534612597)}{\node[red, scale=0.7]  at \Point {$\times$};
}


\draw [dotted, gray] (-1,-0.1) grid (22,0.7);
\node at (0,0.1) {$-$};\node at (-1,0.1) {0.1};
\node at (0,0.2) {$-$};\node at (-1,0.2) {0.2};
\node at (0,0.3) {$-$};\node at (-1,0.3) {0.3};
\node at (0,0.4) {$-$};\node at (-1,0.4) {0.4};
\node at (0,0.5) {$-$};\node at (-1,0.5) {0.5};
\node at (0,0.6) {$-$};\node at (-1,0.6) {0.6};
\node at (2,0) {$|$};\node at (4,0) {$|$};\node at (6,0) {$|$};\node at (8,0) {$|$};\node at (10,0) {$|$};\node at (12,0) {$|$};\node at (14,0) {$|$};\node at (16,0) {$|$};\node at (18,0) {$|$};\node at (20,0) {$|$};
\node at (2,-0.05) {2};\node at (4,-0.05) {4};\node at (6,-0.05) {6};\node at (8,-0.05) {8};\node at (10,-0.05) {10};\node at (12,-0.05) {12};\node at (14,-0.05) {14};\node at (16,-0.05) {16};\node at (18,-0.05) {18};\node at (20,-0.05) {20};
\node[red] at (10,0.6){$\times$};
\node[red] at (13,0.6){$s_n(T_{0.5,-0.5})$};
\node[blue] at (10,0.5){\textbullet};
\node[blue] at (13,0.5){$\widetilde{s}_n(T_{0.5,-0.5})$};
\end{tikzpicture}
\caption{Singular values of $T_{0.5,-0.5}$.}\label{fig1}
\end{figure}

\begin{figure}[h!]
\begin{tikzpicture}[x=1cm,y=10cm]
 \draw[->, thin, draw=gray] (-1,0)--(16,0) node [right] {$n$}; 
 \draw[->, thin, draw=gray] (0,-0.1)--(0,1) node [above] {$s_n(R_{0.5,-0.5})$}; 
\foreach \Point in {(1,0.9577979850),(2,0.3386327251),(3,0.1843283081),(4,0.1197247481),(5,0.08566805611),(6,0.06516989832),(7,0.05171623009),(8,0.04232909063),(9,0.03547399944),(10,0.03028823171),(11,0.02625335985),(12,0.02304103852),(13,0.02043425886),(14,0.01828444850),(15,0.01648682509)}{\node[blue, scale=0.7]  at \Point {\textbullet};
}
\foreach \Point in {(1,0.3250006690),(2,0.1767674033),(3,0.1153593192),(4,0.082923042520),(5,0.063322021380),(6,0.05040561060),(7,0.04136102976),(8,0.03473553702),(9,0.02970998892),(10,0.02579068699),(11,0.02266406397),(12,0.02012233453),(13,0.01802289947),(14,0.01626496862),(15,0.01477548936)}{\node[red, scale=0.7]  at \Point {$\times$};
}


\draw [dotted, gray] (-1,-0.1) grid (16.5,1.1);
\node at (0,0.1) {$-$};\node at (-0.5,0.1) {0.1};
\node at (0,0.2) {$-$};\node at (-0.5,0.2) {0.2};
\node at (0,0.3) {$-$};\node at (-0.5,0.3) {0.3};
\node at (0,0.4) {$-$};\node at (-0.5,0.4) {0.4};
\node at (0,0.5) {$-$};\node at (-0.5,0.5) {0.5};
\node at (0,0.6) {$-$};\node at (-0.5,0.6) {0.6};
\node at (0,0.7) {$-$};\node at (-0.5,0.7) {0.7};
\node at (0,0.8) {$-$};\node at (-0.5,0.8) {0.8};
\node at (0,0.9) {$-$};\node at (-0.5,0.9) {0.9};
\node at (1,0) {$|$};\node at (2,0) {$|$};\node at (3,0) {$|$};\node at (4,0) {$|$};\node at (5,0) {$|$};\node at (6,0) {$|$};\node at (7,0) {$|$};\node at (8,0) {$|$};\node at (9,0) {$|$};\node at (10,0) {$|$};
\node at (11,0) {$|$};\node at (12,0) {$|$};\node at (13,0) {$|$};\node at (14,0) {$|$};\node at (15,0) {$|$};
\node at (1,-0.05) {1};\node at (2,-0.05) {2};\node at (3,-0.05) {3};\node at (4,-0.05) {4};\node at (5,-0.05) {5};\node at (6,-0.05) {6};\node at (7,-0.05) {7};\node at (8,-0.05) {8};\node at (9,-0.05) {9};\node at (10,-0.05) {10};\node at (11,-0.05) {11};\node at (12,-0.05) {12};\node at (13,-0.05) {13};\node at (14,-0.05) {14};\node at (15,-0.05) {15};
\node at (-0.3,-0.05) {$0$};
\node[red] at (11,0.8){$\times$};
\node[red] at (12.5,0.8){$s_n(R_{0.5,-0.5})$};
\node[blue] at (11,0.7){\textbullet};
\node[blue] at (12.5,0.7){$\widetilde{s}_n(R_{0.5,-0.5})$};

\end{tikzpicture}
\caption{Singular values of $R_{0.5,-0.5}$.}\label{fig2}
\end{figure}
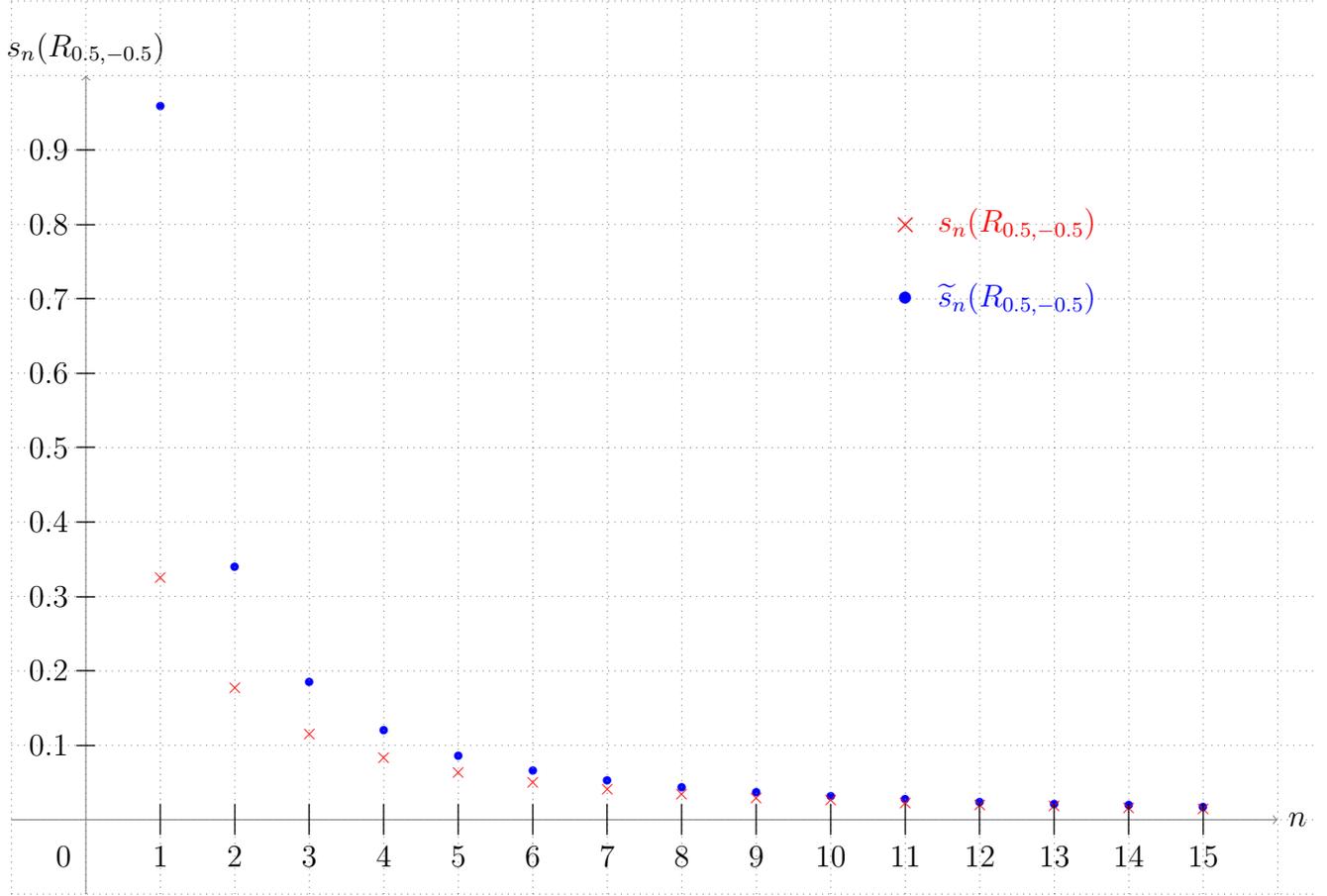
\vskip0.5cm
We claim that Professor J. Benameur proposed the decomposition of the previous integral into two parts using $r_0$. Our earlier proof has two disadvantages, firstly it is based on an assumption (true for $\alpha$ integer) and secondly, it doesn't give exact behavior near infinity. For its originality, we give this proof:
\begin{rem}
    There is a partial proof of the asymptotic result based on the properties of the hypergeometric functions.
\end{rem}
\begin{proof}
For every $n\in\mathbb N$,
     \begin{align*}
         s_{n+1}^2(T_{\alpha,\beta})&\approx \frac{(\alpha+\beta+3)_n}{(\beta+2)_n}\int_0^1\left(1-\frac{h(r)}{h(1)}\right)^2r^{\beta+n}(1-r)^\alpha dr\\
         s_n(R_{\alpha,\beta})&\approx \frac{(\alpha+\beta+2)_n}{(\beta+1)_n}\int_0^1\left(1-\frac{h(r)}{h(1)}\right)r^{\beta+n}(1-r)^\alpha dr
     \end{align*}
     where $s_n\approx \sigma_n$ means that there exist two constants $0<c_1<c_2$ such that $c_1\sigma_n\leq s_n\leq c_2\sigma_n$ for $n$ large enough. i.e. $\sigma_n\lesssim s_n$ and $s_n\lesssim \sigma_n$.\\
     We assume that there exist $\delta_T$ and $\delta_R$ such that $$s_n(T_{\alpha,\beta})\underset{n\to+\infty}{\sim}\frac{a_{\alpha,\beta}}{n^{\delta_T}}\quad \text{and}\quad  s_n(R_{\alpha,\beta})\underset{n\to+\infty}{\sim}\frac{b_{\alpha,\beta}}{n^{\delta_R}}$$
     for some positive constants $a_{\alpha,\beta}$ and $b_{\alpha,\beta}$. (This assumption is true for $\alpha$ integer but we don't know if it is for arbitrary value of $\alpha$ unless we use the previous proof). To conclude the result, it suffices to show that $\delta_T=\delta_R=\alpha+1$. To this aim, we study the following series:
     $$\sum_{n\geq0}n^{2\alpha+1-\varepsilon}s^2_{n+1}(T_{\alpha,\beta})\approx \sum_{n\geq0} \frac{(3\alpha+\beta+3-\varepsilon)_n}{(\alpha+\beta+3)_n}s^2_{n+1}(T_{\alpha,\beta})$$
     for every $0\leq \varepsilon<\alpha.$ It is not hard to see that
     $$\begin{array}{l}
         \ds\sum_{n=0}^{+\infty} \frac{(3\alpha+\beta+4-\varepsilon)_n}{(\alpha+\beta+3)_n}s^2_{n+1}(T_{\alpha,\beta})\\
         \ds\approx \int_0^1\left(1-\frac{h(r)}{h(1)}\right)^2r^{\beta}(1-r)^\alpha \ _2F_1\left(\left.
\begin{array}{c}
1,\ 3\alpha+\beta+4-\varepsilon\\
 \beta+2
\end{array}\right|r\right)dr\\
\ds\approx \int_0^1\left(1-\frac{h(r)}{h(1)}\right)^2\frac{r^{\beta}}{(1-r)^{2\alpha+3-\varepsilon}}\ _2F_1\left(\left.
\begin{array}{c}
\beta,\ -(3\alpha+2-\varepsilon)\\
 \beta+2
\end{array}\right|r\right)dr.
\end{array}$$
The last equality is a consequence of a formula on hypergeometric functions that can be found on page  47 in \cite{Ma-Ob-So}. Using a formula \cite[p.40]{Ma-Ob-So}, we deduce that the function $\ _2F_1\left(\left.
\begin{array}{c}
\beta,\ -(3\alpha+2-\varepsilon)\\
 \beta+2
\end{array}\right|r\right)$   is bounded on $\overline{\mathbb D}$ and its limit at 1 is
$$\ _2F_1\left(\left.
\begin{array}{c}
\beta,\ -(3\alpha+2-\varepsilon)\\
 \beta+2
\end{array}\right|1\right)=\frac{\Gamma(\beta+2)\Gamma(3\alpha+4-\varepsilon)}{\Gamma(3\alpha+\beta+4-\varepsilon)}.$$
Moreover,
$$\lim_{r\to1^-}\left(1-\frac{h(r)}{h(1)}\right)^2\frac{r^{\beta}}{(1-r)^{2\alpha+2}}=\frac{1}{((\alpha+1)h(1))^2}.$$
It follows that the series $\sum_{n\geq0}n^{2\alpha+1-\varepsilon}s^2_{n+1}(T_{\alpha,\beta})$ converges for every $0<\varepsilon<\alpha$ and diverges for $\varepsilon=0$. This gives $\delta_T=\alpha+1.$ \\
With the same argument, we can see that
$$\begin{array}{l}
         \ds\sum_{n=0}^{+\infty} \frac{(2\alpha+\beta+2-\varepsilon)_n}{(\alpha+\beta+2)_n}s_n(R_{\alpha,\beta})\\
         \ds\approx \int_0^1\left(1-\frac{h(r)}{h(1)}\right)\frac{r^{\beta}}{(1-r)^{\alpha+2-\varepsilon}}\ _2F_1\left(\left.
\begin{array}{c}
\beta,\ -(2\alpha+1-\varepsilon)\\
 \beta+1
\end{array}\right|r\right)dr
\end{array}$$
converges for every $0<\varepsilon<\alpha$ and diverges for $\varepsilon=0$. Thus we conclude again that $\delta_R=\alpha+1$.
\end{proof}
\begin{rem}
    For every $\alpha>-1$ and $-\frac{1}{2}<\beta\leq0,$ the operators $T_{\alpha,\beta}$ and $R_{\alpha,\beta}$ are compact.
\end{rem}
    Indeed, we have
    \begin{align*}        T_{\alpha,\beta}f(z)&=\sum_{n=0}^{+\infty}s_n(T_{\alpha,\beta})\langle f,\psi_{n}\rangle_{\alpha,\beta} e_n(z)\\
        R_{\alpha,\beta}f(z)&=-\sum_{n=0}^{+\infty}s_n(R_{\alpha,\beta})\langle f,e_{n+1}\rangle_{\alpha,\beta} e_n(z)
    \end{align*}
    and their singular values converge to zero. Hence we deduce that the operators are both compact as limits of sequences of operators with finite ranks.
 \section{Concluding remarks}
 We collect here some properties of $T_{\alpha,\beta}$ and $R_{\alpha,\beta}$.
 \begin{enumerate}
     \item Since $T_{\alpha,\beta}^*T_{\alpha,\beta}$ and $R_{\alpha,\beta}^*R_{\alpha,\beta}$ are self adjoint operators, then $$\|T_{\alpha,\beta}\|^2=\|T_{\alpha,\beta}^*T_{\alpha,\beta}\|=\sup_{n\geq0} s_n(T_{\alpha,\beta})^2$$ and $$\|R_{\alpha,\beta}\|=\sup_{n\geq0}s_n(R_{\alpha,\beta}).$$
     \item Since for every $p>\frac{1}{\alpha+1}$, the series $\sum_{n\geq} n^{-p(\alpha+1)}$ converges,  then both operators  $T_{\alpha,\beta}$ and $R_{\alpha,\beta}$ are in the $p^{th}-$Schatten class. In particular, they are Hilbert-Schmidt operators whenever $\alpha>-\frac{1}{2}$.
     \item Both operators $T_{\alpha,\beta}$ and $R_{\alpha,\beta}$ are neither hyponormal nor paranormal. Indeed, a simple computation shows that  $$T_{\alpha,\beta}^2f(z)=\sum_{n=0}^{+\infty}s_n(T_{\alpha,\beta})s_{n+1}(T_{\alpha,\beta})\langle f,e_{n+1}\rangle_{\alpha,\beta}\langle \psi_{n+1},e_{n+1}\rangle_{\alpha,\beta}\psi_{n+1}(z).$$
 So we obtain $T_{\alpha,\beta}e_0\neq0$ and $T_{\alpha,\beta}^2e_0=0$. Hence we don't have $\|T_{\alpha,\beta}f\|^2\leq \|T_{\alpha,\beta}^2f\|$ for every $f\in L^2(\mathbb D,\mu_{\alpha,\beta})$. Thus $T_{\alpha,\beta}$ is not paranormal. The non-hyponormality of $T_{\alpha,\beta}$ can be deduced from the fact that $[T_{\alpha,\beta}^*,T_{\alpha,\beta}]:=T_{\alpha,\beta}^*T_{\alpha,\beta}-T_{\alpha,\beta}T_{\alpha,\beta}^*$ is not positive because $$[T_{\alpha,\beta}^*,T_{\alpha,\beta}]f(z)=\sum_{n=0}^{+\infty}|s_n(T_{\alpha,\beta})|^2\left(\langle f,e_n\rangle_{\alpha,\beta} e_n(z)-\langle f,\psi_n\rangle_{\alpha,\beta} \psi_n(z)\right).$$
The same conclusions can be deduced for $R_{\alpha,\beta}$ from the fact that
$$R_{\alpha,\beta}^2f(z)=\sum_{n=0}^{+\infty}s_n(R_{\alpha,\beta})s_{n+1}(R_{\alpha,\beta})\langle f,e_{n+2}\rangle_{\alpha,\beta} e_n(z)$$
 and  $$[R_{\alpha,\beta}^*,R_{\alpha,\beta}]f(z)=\sum_{n=0}^{+\infty}|s_n(R_{\alpha,\beta})|^2\left(\langle f,e_{n+1}\rangle_{\alpha,\beta} e_{n+1}(z)-\langle f,e_n\rangle_{\alpha,\beta} e_n(z)\right).$$
\item Open problems:
\begin{enumerate}
    \item The conditions $-\frac{1}{2}<\beta\leq0$ and $\alpha>-\frac{2\beta+1}{3\beta+2}$ are sufficient for the compactness of the Cauchy transform on $L^2(\mathbb D,\mu_{\alpha,\beta})$. Can one improve these conditions (namely prove the result for any $\alpha,\beta>-1$)?
    \item Find the singular values of $\mathcal C\mathbb P$ and $\mathbb P\mathcal C\mathbb P$ when $\mathbb P$ is the orthogonal projection on the subspace of poly-analytic or harmonic functions in $\mathbb D$.
\end{enumerate}

 \end{enumerate}
\section*{Conflict of interest statement}
The authors declare that there is no conflict of interest.
 \section*{Acknowledgements}
The authors would like to thank Professor Jamel Benameur for his help in proving the asymptotic result.

\end{document}